\documentclass[aps,prb,amsmath,amssymb,floatfix,12pt]{article}%{revtex4-1}
\usepackage[T1]{fontenc}
\usepackage[latin9]{inputenc}
\usepackage{color}
\usepackage{prettyref}
\usepackage{units}
\usepackage{amsthm}
\usepackage{amsbsy}
\usepackage{amstext}
\usepackage{amssymb}
\usepackage{fancyhdr}
\usepackage{stmaryrd}
\usepackage{mathrsfs}
\usepackage{mathtools}
\usepackage{euscript}
\usepackage{upgreek}
\usepackage{setspace}
\usepackage[headings]{fullpage}

\usepackage[unicode=true,pdfusetitle,
 bookmarks=true,bookmarksnumbered=false,bookmarksopen=false,
 breaklinks=false,pdfborder={0 0 0},backref=false,colorlinks=true]
 {hyperref}
\usepackage{calligra,mathrsfs}

\pdfpageheight\paperheight
\pdfpagewidth\paperwidth

\theoremstyle{plain}
\newtheorem{lem}{Lemma}
\newtheorem*{prop*}{Proposition}
\newtheorem{prop}{Proposition}
\newtheorem{thm}{Theorem}
\newtheorem*{thm*}{Theorem}
\newtheorem*{cor*}{Corollary}
\newtheorem{cor}{Corollary}

\theoremstyle{definition}
\newtheorem{defin}{Definition}
\newtheorem*{defin*}{Definition}

\numberwithin{equation}{section}

\newcommand{\barh}{{\bar h}}
\newcommand{\alphainv}{{\alpha^{-1}}}
\newcommand{\cusp}{{\alphainv\infty}}
\newcommand{\Z}{{\mathbb Z}}
\newcommand{\Complex}{{\mathbb C}}
\newcommand{\bbH}{{\mathbb H}}

\newcommand{\R}{{\mathbb R}}
\newcommand{\Q}{{\mathbb Q}}
\newcommand{\bbR}{{\mathbb R}}
\newcommand{\SL}{{\operatorname{SL}_2(\Z)}}

\newcommand{\PSLR}{{\operatorname{PSL}_2(\R)}}

\newcommand{\GL}{{\operatorname{GL}_2(\Z)}}
\newcommand{\Rad}{\vec R_{\Gamma,w,\rho,\hat iq^{n_i}}(\tau)}
\newcommand{\Radnovec}{R_{\Gamma,w,\rho,\hat iq^{n_i}}(\tau)_j}

\newcommand{\Radalpha}{\vec R_{\Gamma,w,\rho,\hat i(\alpha q)^{n_i}}(\tau)}
\newcommand{\Ponalpha}{\vec P_{\Gamma,2-w,\bar\rho,\hat i(\alpha q)^{-n_i}}(\tau)}

\newcommand{\Radj}{R_{\Gamma,w,\rho,\hat i(\alpha q)^{n_i}}(\tau)_j}
\newcommand{\Ponj}{P_{\Gamma,2-w,\bar\rho,\hat i(\alpha q)^{-n_i}}(\tau)_j}

\newcommand{\sm}[4]{\left( \begin{smallmatrix} #1&#2\\ #3&#4 \end{smallmatrix} \right)}

\newcommand{\Ginf}{\Gamma_\infty}
\newcommand{\GG}{\Ginf\backslash\Gamma}
\newcommand{\GGstar}{\GG^*}
\newcommand{\GGG}{\GG/\Ginf}
\newcommand{\GGstarG}{\GGstar/\Ginf}
\newcommand{\GaG}{\Ginf\backslash\alpha\Gamma}

\newcommand{\GaGK}{\Ginf\backslash(\alpha\Gamma)_{K,K^2}}

\DeclareMathOperator*{\im}{Im}
\DeclareMathOperator*{\re}{Re}
\DeclareMathOperator{\Hom}{\mathscr{H}\text{\kern -3pt {\calligra\large om}}\,}
\DeclareMathOperator*{\intint}{\iint}
\DeclareMathOperator{\Kl}{Kl}
\DeclareMathOperator{\e}{e}
\DeclareMathOperator{\rad}{rad}

\begin{document}

\title{Vector-Valued Rademacher Sums and Automorphic Integrals}
\author{Daniel Whalen}
%\affiliation{Stanford Institute for Theoretical Physics, Stanford University, Stanford, California 94305, USA}
%\affiliation{SLAC National Accelerator Laboratory, 2575 Sand Hill Road, Menlo Park, CA 94025}

\date{\today}
\maketitle

\begin{abstract}
We present bases for certain spaces of meromorphic vector-valued rational-weight mock modular forms constructed using Rademacher sums.
\end{abstract}

\section{Introduction}

Scalar-valued modular forms have played a fundamental role in number theory for decades.  Generalizing the multiplier systems of modular forms to representations of higher dimension is a natural  extension that has appeared historically, but not in as great detail as the scalar-valued theory.  It is only recently that mathematicians have started to explore this concept rigorously in~\cite{Bantay:2007vl,Bantay:2003vm,Bantay:2013hf,Bantay:2010kc,Freitag:2012wc,Knopp:2004uq,Knopp:2003iv}.

The recent development in the theory of vector-valued modular forms has followed the discovery that they appear naturally in certain contexts.   Zhu~\cite{Zhu:1996vq} showed that the graded dimensions of the irreducible modules of  suitable vertex operator algebras are the Fourier expansions of vector-valued modular forms.  This theorem has a conjectural extension, that the graded traces of symmetries of those vertex operator algebras also form vector-valued modular forms with different multiplier systems over certain Fuchsian groups\cite{Dong:1997jw}.  The vector-valued modular forms  appear in related situations in string theory and two-dimensional conformal field theories, where they provide expressions for the torus partition functions\cite{Manschot:2007vx}.  The theory of Jacobi forms can be naturally expressed by decomposing the forms into vector-valued modular forms. 

The conjectural extension of Zhu's result, that the graded traces of vertex operator algebras are modular forms has played an important role in the study of Moonshine.  For Monstrous Moonshine, it has been proven that the graded traces of symmetries of the Monstrous vertex operator algebra were scalar-valued modular functions, and were in fact, the Hauptmodulen for certain genus 0 modular groups.\cite{Borcherds:1992bi}  Recently, the Hauptmodul property of Monstrous Moonshine has been rephrased in terms of Rademacher sums, a certain generalization of Poincar\'e series~\cite{Duncan:2011uj}.  In Monstrous Moonshine, the Rademacher sums that appear give scalar-valued modular forms, but certain discoveries by Cheng, Duncan, and Harvey suggest that the vector-valued generalization also play a role in interpreting Moonshine~\cite{Cheng:2012ue,Cheng:2013vr}.

%Cheng, Duncan, and Harvey realized that Rademacher sums play a role in  the Moonshine for the sporadic Mathieu group $M_{24}$ suggested by the elliptic genera of $K3$ manifolds.  In particular, certain Rademacher sums provide the Fourier coefficients of the mock-modular forms associated to the twining genera of $M_{24}$.\cite{Cheng:2012jl}  When $M_{24}$ moonshine was generalized to arbitrary Niemeier lattices, the Rademacher sums that appeared were to be vector-valued~\cite{Cheng:2013vr}.  Cheng and Duncan expanded the theory of Rademacher sums to include sums of weight $\tfrac{1}{2}$ over certain subgroups and multiplier systems associated to these moonshines, but have left the general case of vector-valued Rademacher sums undiscussed.

In this paper, we prove that many of the results associated with scalar-valued Rademacher sums generalize to vector-valued Rademacher sums of weight $w\leq 0$.  In particular, we prove the following results: (1) in Theorems~\ref{Kllimit} and~\ref{cor00}, that the asymptotic behavior of the matrix-valued Kloosterman-Selberg zeta function is similar to that of the scalar case, (2) in Theorem~\ref{convthm} that a certain generalization of vector-valued Rademacher sums converges for weight $w\leq 0$, and (3) in Theorem~\ref{basis} that the Rademacher sums comprise a basis for the space of vector-valued automorphic integrals for that weight and multiplier system. 

These results are related to previous works.  In the case of $\Gamma = \SL$, the vector-valued Rademacher sums provide a separate way of constructing the modular forms that are solutions to the hypergeometric equations presented in~\cite{Bantay:2007vl}.  In doing so, they also provide a basis for the space of possible graded traces of vertex operator algebras, which are vector-valued modular functions under $\SL$.  Lastly, the Rademacher sum construction permits an expression for the asymptotic behavior of the Fourier coefficients of vector-valued modular functions and therefore the asymptotic behavior of the dimensions of irreducible representations of vertex operator algebras.

The paper is organized in the following way:  In section~\ref{s1}, we introduce the notation used and the definitions of modular forms and modular groups.  In section~\ref{s15}, we motivate and define the Rademacher sum.  In section~\ref{s2}, we introduce automorphic integrals and shadows.  In section \ref{kloostermansection} we define a generalization of the Kloosterman sum and the Kloosterman-Selberg zeta function.
In section~\ref{coefficients}, we provide explicit expressions for the Fourier coefficients in the decompositions of Rademacher sums and their shadows.   In section~\ref{s3} we prove that the Rademacher sums generate the space of automorphic integrals up to possible constant functions.  Many of the details of the proofs are confined to the appendices.  In Appendix~\ref{kloosterman}, we summarize results on the asymptotic behavior and convergence of the Kloosterman-Selberg zeta function.  In Appendix~\ref{convergence}, we prove the convergence of the vector-valued Rademacher sums of weight $w\leq0$.  In Appendix~\ref{proofoflemma}, we provide a limit on the asymptotic behavior of the matrix-valued Kloosterman-Selberg zeta function that is used in the convergence proof.  In Appendix~\ref{otherproof}, we prove similar bounds for the Kloosterman-Selberg zeta function for a particular choice of parameters.

% SECTION: MODULAR GROUPS AND MODULAR FORMS
\section{Modular Groups and Modular Forms}\label{s1}

%Introduction
The modular groups used in this paper will be Fuchsian groups, discrete subgroups of $\PSLR$, that have a fundamental domain of finite volume, a cusp at $\infty$, and possibly other cusps.
%It would be possible to generalize many of the results to more arbitrary Fuchsian groups with arbitrary complex weight, but in practice, the modular groups that are relevant in the study of vertex operator algebras and moonshine satisfy these restrictions.  Similarly, 
We will make assumptions on the allowable multiplier systems for the modular functions that will remove some generality, but not reduce the theory's applicability to vertex operator algebras or moonshine.

%Modular Groups
Let $\Gamma$ be a Fuchsian group with finite volume and a cusp at $\infty$.  
Given a cusp $\alpha^{-1}\infty$ for $\alpha\in\PSLR$ we say that the width of $\alpha^{-1}\infty$ is $\barh$, where $\bar h$ is the minimal positive real number such that $\alpha^{-1}T^{\bar h}\alpha=\alpha^{-1}\sm{1}{\bar h}{0}{1}\alpha\in\Gamma$.
In general, write  $h$ for the width of the cusp at $\infty$.  Write $\Ginf$ for the subgroup of $\Gamma$ that is generated by $T^h$ and $-I$, or equivalently, the subgroup that fixes $\infty$.

 $\Gamma$ admits an action on the upper half plane, $\bbH$, by
\begin{equation*}
\gamma(\tau) = \frac{a\tau+b}{c\tau+d}
\end{equation*}
for $\gamma=\sm{a}{b}{c}{d}\in\Gamma$, $\tau\in\bbH$.  For an arbitrary rational weight $w$, define
\begin{equation*}
j_w(\gamma,\tau) = \left(\frac{\partial\gamma(\tau)}{\partial\tau}\right)^{-w/2} = \exp\left(\frac{w}{2}\ln\left((c\tau+d)^2\right)\right).
\end{equation*}
In general, we take the logarithm along the principal branch: for $z\in\Complex$, $-\pi i <\im\ln z\leq \pi i$.  The choice of the branch in the logarithm introduces a phase into the composition equation
\begin{equation*}
j_w(\beta\alpha,\tau) = \omega_w(\alpha,\beta)^{-1}j_w(\beta,\alpha\tau)j_w(\alpha,\tau),
\end{equation*} 
where $\alpha,\beta\in\Gamma$, and $\omega_w(\alpha,\beta)$ is $\tau$-independent and of unit magnitude.  It is easy to see that if $w$ is an even integer, then $\omega_w(\alpha,\beta) = 1$ for all $\alpha,\beta\in\Gamma$.  This correction is also introduced into the multiplier systems for modular forms.

%Multiplier systems
\begin{defin}{}
Let $\Gamma$ be a Fuchsian group with finite volume and a cusp at $\infty$.  A normal multiplier system of weight $w\in\Q$ is a map $\rho:\Gamma\to\operatorname{U}(d)$, the group of $d\times d$ unitary matrices, such that
\begin{itemize}
\item $\rho$ has finite image in $\operatorname{U}(d)$.
\item $\rho$ satisfies the consistancy condition: for any $\alpha,\beta\in\Gamma$ and $\tau\in\bbH$, $\rho(\alpha\beta)\omega_w(\beta,\alpha)= \rho(\alpha)\rho(\beta)$
\item for all $\gamma = \sm{1}{n}{0}{1}\in\Gamma$, $\rho(\gamma)$ is diagonal.
%\item  $\rho(-I)$ is a permutation matrix or $-1$ times a permutation matrix.
\end{itemize}
\end{defin}

If $w$ is an even integer, the consistency condition simplifies to $\rho(\alpha\beta) = \rho(\alpha)\rho(\beta)$ and $\rho$ is a unitary representation. 
%These assumptions that we have made on the the available multiplier systems do not substantially restrict the relevance of our results.   In particular, all of these assumptions are true for the vector-valued modular forms that are the characters of vertex operator algebras~\cite{Bantay:2003vm,Bantay:2010kc}.

For the rest of these notes, we will assume that all multiplier systems are normal.  While working in these systems, we require some additional notation.  We use the following: given a function $f:\Complex\to\Complex$, and $v\in\Complex^d$, write $\left[f(\vec v)\right]$ for the $d\times d$ diagonal matrix whose diagonal entries are $f(v_i)$ as $i$ ranges from $1$ to $d$.  Now, since $\rho(T^h)$ is diagonal and unitary, all of its coefficients on the diagonal are roots of unity. Define $e(x) = e^{2\pi i x}$ and write $\rho(T^h)=\left[e(\vec\mu)\right]$ for $0\leq\mu_i<1$.

This notion generalizes to cusps not at infinity.  In general, given a cusp $\alpha^{-1}\infty$ of width $\barh$ for $\alpha\in\PSLR$ and $\rho$, we fix some $\rho_\alpha$ that diagonalizes $\rho(\alphainv T^\barh\alpha)$ and let $\rho_\alpha\rho(\alphainv T^\barh\alpha)\rho_\alpha^{-1}=[e(\vec\nu)]$ for $0\leq\nu_i<1$.

% Spaces of Modular forms
\begin{defin}{}
A holomorphic modular function of weight $w\in\Q$ over $\Gamma$ with multiplier system $\rho$ is a
vector-valued holomorphic function $\vec f: \bbH \to\Complex^d$ that satisfies for any $\gamma\in\Gamma$:
\begin{equation*}
\vec f(\gamma\tau) = j_w(\gamma,\tau)\rho(\gamma)\vec f(\tau)
\end{equation*}
\end{defin}

A function is called weakly holomorphic if it has a pole of finite order in each of its components as $\tau$ approaches any cusp, but otherwise satisfies the above constraints.  %To avoid additional complexity, for the remainder of this paper we will only consider weakly holomorphic modular functions whose only poles are at infinity.  The results when poles exist at other cusps are straightforward generalizations of the equations here.

Since the modular groups that we are using have a finite width at $\infty$, the weakly holomorphic forms transform with a component-dependent phase under $\sm1h01$.  These modular functions can then be expressed in terms of their Fourier series around $\infty$.  Write %$e(x) = \exp(2\pi i x)$ and 
$q=e(\tau)$.  Then

\begin{equation*}
\vec f(\tau) = \left( \begin{array}{c}
q^{(\mu_1-m_1)/h}(a_{10}+a_{11}q +\ldots)\\
q^{(\mu_2-m_2)/h}(a_{20}+a_{21}q +\ldots)\\
\vdots\\
q^{(\mu_d-m_3)/h}(a_{d0}+a_{d1}q +\ldots) \end{array} \right),
\end{equation*}
for $h$ and $\mu_i$ as defined above and for some integers $m_i$.

We use a standard notation to refer to the commonly used vector spaces of modular forms and weakly holomorphic modular forms.
Denote by $\mathcal M_w(\rho)$ the vector space of weakly holomorphic forms of weight $w$ and multiplier system $\rho$ that are holomorphic at cusps inequivalent to $\infty$.  Denote by $\mathcal M_w^m(\rho)\subset \mathcal M_w(\rho)$ the subspace with pole at infinity and equivalent cusps of degree at most $m$. Denote by $M_w(\rho)$ the space of modular forms of weight $w$ and multiplier system $\rho$, and by $S_w(\rho)$ the space of cusp forms, that is, the subspace of $M_w (\rho)$ that tends to $\vec 0$ at each of the cusps of $\Gamma$.  We will not explicitly consider the spaces of forms with cusps other than $\infty$ when counting dimensions, but the theorems generalize naturally to that case.

%SECTION: POINCARE SERIES AND RADEMACHER SUMS
\section{Poincar\'e Series and Rademacher Sums}\label{s15}
Scalar Poincar\'e series are a standard construction for modular forms.  The expression of a Poincar\'e series extends naturally to vector-valued modular forms\cite{Knopp:2004uq}.  We describe the construction below.

% Poincare sum definition
Let $\hat{i}$ represent the $i$th  unit vector in the standard basis for $\R^d$.  Fix $\Gamma$, a Fuchsian group with finite volume and a cusp at $\infty$, and $\rho$ and normal multiplier system of weight $w\in\Q$.  Fix $0<h n_i\in\Z+\mu_i$.    For the sake of ease of notation, we write $\gamma q$ for $e(\gamma\tau)$.  Using this notation, the vector-valued Poincar\'e series is given by

\begin{equation*}
\vec P_{\Gamma,w,\rho,n_i}(\tau) = \sum_{\gamma\in\GG}j_w(\gamma,\tau)^{-1}\rho(\gamma)^{-1}\cdot \hat i\cdot(\gamma q)^{n_i}.
\end{equation*}
The sum is over representatives of the cosets in $\GG$.  Observe that this sum is well-defined; choosing a different element of the coset introduces a phase in the $\rho(\gamma)^{-1}$ term that cancels with the phase from the $q_\gamma$ term.

This series is absolutely convergent for $w>2$, but generally divergent for $w\leq2$, which is the region that is applicable to vertex operator algebras.  Rademacher found a solution to this problem by considering a related series where the summands were regularized to obtain an expression for the $j$-invariant~\cite{Rademacher:1939iz},
\begin{equation*}
j(\tau) + 12 = e(-\tau) +  \lim_{K\to\infty}\sum_{\gamma\in\GG^*_{K,K^2}} e(-\gamma\tau) - e(-\gamma\infty).
\end{equation*}

In this equation, the sum is over cosets of $\Gamma^* = \Gamma\smallsetminus\Ginf$ to remove the component where the summand would be divergent. The $K, K^2$ represents a specific order of the summation in which we sum over elements of increasing $c$ as defined below.  A technique by Niebur can be used to generalize this expression by regularizing the sums of more general Poincar\'e series~\cite{Niebur:1974vw}.  In Rademacher's case, namely $\SL$, the sum was still a modular function, but in general, we sacrifice modularity to the regularization.  We will see later that while the sums are not always modular functions, they do satisfy some automorphic property.  Niebur's expressions can be readily generalized to the vector-valued case, as we see below.

%\rad definition
The regularization factor is defined in terms of the normalized lower incomplete
gamma function, $\boldsymbol{\gamma}$~\cite{Cheng:2012vk}. For $\gamma =\sm1001$, let $\rad^n _w(\gamma,\tau) = 1$ and otherwise
\begin{align}
\rad^n _w(\gamma,\tau) &= \frac{1}{\Gamma(1-w)}\boldsymbol\gamma\left(1-w,2\pi in(\gamma\tau-\gamma\infty\right))\\
&=e(-n\gamma\tau+n\gamma\infty)\sum^{\infty}_{m=0}\frac{(2\pi in (\gamma\tau-\gamma\infty))^{m+1-w}}{\Gamma(m+2-w)}\label{eq545}\\
&=\frac{1}{\Gamma(1-w)}\int^{2\pi in(\gamma\tau-\gamma\infty)}_0t^{-w}e^{-t}dt,
\end{align}
	%where the various expressions will be useful in different contexts.  

In the specific case that $w = 0$, this expression reduces to Rademacher's regularization~\cite{Rademacher:1939iz},
\begin{equation*}
\rad^n _0(\gamma,\tau)=1-e(-n\gamma\tau +n\gamma\infty).
\end{equation*}

%\Rad definition
The regularization factor will allow us to define a convergent Rademacher sum.  The sum will not be absolutely convergent in most cases, so we choose the order of summation carefully.  Let $\Gamma_{K,K^2}$ be the set of elements $\sm abcd\in\Gamma$ such that $0\leq c<K$ and $|d|<K^2$.  We then denote by $\GG_{K,K^2}$, the subset of $\GG$ that consists of cosets $\Ginf\gamma$ such that $\Ginf\gamma\cap\Gamma_{K,K^2}$ is nontrivial.

Fix $i$.  Take $n_i<0$ with $h n_i\in\Z+\mu_i$, and $n_i<0$.  Then,
\begin{equation*}
%\Rad
\vec R_{\Gamma,w,\rho,\hat iq^{n_i}}(\tau)
=\vec\Delta+\lim_{K\to\infty}\sum_{\gamma\in\GG_{K,K^2}}j_w(\gamma,\tau)^{-1}\rho(\gamma)^{-1}\rad^{n_i}_w(\gamma,\tau)\cdot \hat i\cdot(\gamma q)^{n_i},
\end{equation*}
where $\vec\Delta$ is a constant given in terms of the Kloosterman-Selberg zeta function, $\Kl$ discussed in section~\ref{kloostermansection}.  In the above case, when the cusp is at infinity, $\alpha$ is taken to be the identity element in the definition of $\Delta$.
\begin{equation*}
\Delta_j =
\begin{cases}
0 &\mbox{if } \mu_j\neq 0\\
-\frac{1}{2h}(2\pi i)^{2-w}\frac{1}{\Gamma(2-w)}
\Kl^{\alpha^{-1}\infty}_{n_i,0}(1-\frac{w}{2})_{ji}
(-n_i)^{1-w}&\mbox{if } \mu_j = 0
\end{cases}.
\end{equation*}

This definition generalizes to other cusps.  Let $\cusp$ be a cusp with width $\barh$ and corresponding $\vec\nu$ and $\rho_\alpha$.  Let $(\alpha \Gamma)_{K,K^2}$ be defined analogously to before.  As before let $\Gamma_\infty \backslash (\alpha \Gamma)_{K,K^2}$ be the set of cosets of $\alpha\Gamma$ on the left by the group generated by $T^\barh$ with nontrivial intersection with $(\alpha\Gamma)_{K,K^2}$.

Now fix $i$ and $n<0$ with $\barh n\in\Z+\nu_i$.  The Rademacher sum for a pole at at the cusp $\cusp$ is
\begin{multline*}
\Radalpha
=\vec\Delta
+\lim_{K\to\infty}\sum_{\gamma\in\GaGK}
j_w(\gamma,\tau)^{-1}
\omega_w(\alphainv,\gamma)\\
\cdot\rho(\alphainv\gamma)^{-1}\rho_\alpha^{-1}\rad^{n_i}_w(\gamma,\tau)
\cdot \hat i\cdot(\gamma q)^{n_i}.
\end{multline*}
with the constant correction as before, but evaluated at the corresponding cusp.

Convergence of the Rademacher sums in the scalar-valued case has been well-studied.  See~\cite{Niebur:1974vw} for the $w<0$ case and~\cite{Duncan:2011uj} for the $w=0$ case.  The proof of convergence in the vector-valued case is similar for $w\leq 0$.  Due to its complexity, we relegate the proof to Appendix~\ref{convergence}.

% Section automorphic properties
\section{Automorphic Properties of Rademacher Sums}\label{s2}
Inclusion of the regularization term generally removes the modular invariance property of the Poincar\'e sum.  The Rademacher sums, however, still enjoy a similar property.  A non-holomorphic correction term of a particular form can be added to the Rademacher sum to make the sum a non-holomorphic modular function.  Niebur showed for scalar-valued Rademacher sums that these correction terms have a simple expression in terms of the Poincar\'e series.   

We follow the approach of Niebur in his classification of scalar-valued Rademacher sums~\cite{Niebur:1974vw}.  We will omit details of the proofs in this section because they are identical to the scalar-valued case.

%Automorphic Integral definitions
\begin{defin}
An automorphic integral of weight $w$, modular group $\Gamma$ and multiplier system $\rho$ is a holomorphic map  $\vec f:\mathbb H\to\mathbb C^d$, such that for some cusp form $\vec g\in S_{2-w}(\rho)$, we have that for all $\gamma\in\Gamma$,
\begin{equation}\label{autodef}
\vec f(\gamma\tau) = (c\tau+d)^w\rho(\gamma)\left(\vec f(\tau)-p(w,\gamma^{-1}\infty;\vec g)\right),
\end{equation}
where the correction term is given in terms of the shadow, $\vec g$,
\begin{equation*}
p(w,\tau;\vec g) = \frac{1}{\Gamma(1-w)}\int^{i\infty}_{-\overline{\tau}}(z+\tau)^{-w}\overline{\vec g(-\overline z)}{dz}.
\end{equation*}
and where $\vec H (\tau ) = (\vec f(\tau)-p(w, \tau;\vec g))$ is meromorphic at all of the cusps of $\Gamma$.
\end{defin}

This definition is equivalent to the standard definition of a mock modular function whose shadow is a cusp form\cite{Cheng:2012vk}.

The automorphic integrals form a vector space.  We denote the space of automorphic integrals  of weight $w$, multiplier system $\rho$, and maximum pole order of $H(\tau)$ at the cusp $i\infty$ less than or equal to $m$ as $\mathcal J_w(m,\rho)$.  We will show later that the Rademacher sums form a basis for $\mathcal J_w(m,\rho)$, but first we need to know that
\begin{thm}
For $w\leq 0$ and $hn_i\in\Z+\mu_i$ with $n_i<0$, the Rademacher sum $\Radalpha$ is an automorphic integral with shadow given by the dual Poincar\'e series $\Ponalpha$.
\end{thm}
\begin{proof}
The automorphic properties follow through the same logic as in the scalar-valued case with the pole at $\infty$.  See~\cite{Niebur:1974vw} for a rigorous proof and a discussion of the $\vec\Delta$ term, which compensates for the reordering, or~\cite{Cheng:2012jl} for a more accessible sketch.
\end{proof}

% Section Kloosterman
\section{Kloosterman Sums and the Kloosterman-Selberg Zeta Function}\label{kloostermansection}

%Kloosterman sum
The expressions for the Rademacher sums are given in terms of matrix-valued Kloosterman sums.

\begin{defin}\label{Sdef}
Fix $\Gamma$ with a cusp at $\infty$ of width $h$, a multiplier system $\rho$ of weight $w$, a cusp $\cusp$ of width $\barh$ and corresponding $\rho_\alpha$, components $i,j$ and $k_j, n_i$ such that $hk_j\in\Z+\mu_j$ and $\barh n_i\in\Z+\nu_i$.  The Kloosterman sum for a $c'>0$ is given by
\begin{equation*}
S^\cusp_{n_i,k_j}(c')_{ji}
= \sum_{\substack{\gamma\in\GaG/\Ginf\\c(\gamma)=c'}}\e(n_i\gamma\infty-k_j\gamma^{-1}\infty)\{\omega_w(\alpha^{-1},\gamma)\rho^{-1}(\alpha^{-1}\gamma)\rho_\alpha^{-1}\}_{ji}.
\end{equation*}
%\begin{equation*}
%S_{\vec m,\vec n}(c',\rho) = \sum_{\substack{\gamma\in\GGG\\c(\gamma)=c'}}\left[\e\left(\frac{-\vec n\gamma^{-1}\infty}{h}\right)\right]\rho^{-1}(\gamma)\left[\e\left(\frac{\vec m\gamma\infty}{h}\right)\right].
%\end{equation*}
The sum is over the double cosets $\gamma\in\GaG/\Ginf$, where the left $\Ginf$ corresponds to width $\barh$ and the right corresponds to width $h$, further restricted by requiring that $\gamma=\sm abcd$ satisfies $c=c'$.  The $\{M\}_{ji}$ indicates the $(j,i)$ coefficient of the matrix $M$.
\end{defin}

Observe that left and right multiplication by $T^\barh$ and $T^h$ respectively do not change $c$ and leave the summand invariant, so the sum is well-defined.  It is necessarily the case that the number of such double cosets is finite for fixed $c$, and in fact is $O(c)$ because of the finite volume condition on the Fuchsian group~\cite{Iwaniec:2002vu}.

At this cusp at infinity this definition corresponds to the standard definition.  For a scalar multiplier system $\rho$,
\begin{equation*}
S^\infty_{n,k}(c',\rho)=\sum_{\substack{\gamma\in\GGG\\c(\gamma)=c'}}e(n\gamma\infty-k\gamma^{-1}\infty)\rho^{-1}(\gamma)
\end{equation*}

%Frequently we will only desire one component of the Kloosterman Sum.  In those cases, we omit the irrelevant components of $\vec m$ and $\vec n$ and write
%
%\begin{equation*}
%S_{m_i,n_j}(c',\rho) _{ji} = \sum_{\substack{\gamma\in\GGG\\c(\gamma)=c'}}\e\left(\frac{-n_j\gamma^{-1}\infty}{h}\right)\rho^{-1}(\gamma)_{ji}\e\left(\frac{m_i\gamma\infty}{h}\right).
%\end{equation*}

\begin{defin}\label{Kl}
Fix $\Gamma$ with a cusp at $\infty$ of width $h$, a multiplier system $\rho$ of weight $w$, a cusp $\cusp$ of width $\barh$ and corresponding $\rho_\alpha$, components $i,j$ and $k_j, n_i$ such that $hk_j\in\Z+\mu_j$ and $\barh n_i\in\Z+\nu_j$.  The Kloosterman-Selberg zeta function at $s\in\Complex$ is given by
\begin{equation*}
\Kl^\cusp_{n_i,k_j}(s,\rho)_{ji}
= \sum_{c>0}\frac{S^\cusp_{n_i,k_j}(c',\rho)_{ji}}{c^{2s}}
\end{equation*}
\end{defin}
The Kloosterman-Selberg zeta function converges absolutely for $\re\alpha>1$.  The asymptotics of this function will be important for the convergence proof for the Rademacher sums and are discussed more in Appendix~\ref{kloosterman}.

\section{Fourier Coefficients of Rademacher Sums}\label{coefficients}

The Fourier coefficients of the Rademacher sums have been used to understand the dimensions of vertex operator algebras and the graded traces of automorphisms of these spaces.  The standard expressions for the Fourier coefficients permit a generalization to the vector-valued case.

% Bessel J function and definition
%The proof of convergence in Appendix~\ref{convergence} gives the Fourier components of the Rademacher sum as an intermediate step.  Using the Bessel $J$ function gives a clean expression for the coefficients in terms of a sum over $c$ and the matrix-valued Kloosterman sum.  We start by defining the matrix valued Kloosterman Sum, which is a generalization of the scalar valued case.

As we see in Appendix~\ref{convergence}, the expression of vector-valued Rademacher sums as a $q$-series can be found via Lipschitz summation and application of the infinite sum expression for the Bessel $J$ function,

\begin{equation*}
J_\alpha(z) =\sum_{m=0}^\infty\frac{(-1)^m}{m!\Gamma(m+\alpha+1)}\left(\frac{z}{2}\right)^{2m+\alpha} .
\end{equation*}

Using the Kloosterman sums, we find the following expressions for the Rademacher sum and its shadow.

\begin{thm}\label{fouriercoefficients}
Fix $\Gamma$ with a cusp at $\infty$ of width $h$, a multiplier system $\rho$ of weight $w$, a cusp $\cusp$ of width $\barh$ and corresponding $\rho_\alpha$, components $i,j$ and $n_i$ such that $\barh n_i\in\Z+\nu_j$, and $n_i<0$.  Let $\bar\rho$ be the conjugate multiplier system, where $\bar\rho(\gamma) = \overline{\rho(\gamma)}$ with corresponding $\vec\mu'$ and $\vec\nu'$.  Let $\vec\delta_\alpha=1$ if $\cusp=\infty$ and 0 otherwise.  Let $\Delta$ be as above.  The $j$th components of the Rademacher sum and its shadow Poincar\'e sum are given by

\begin{multline*}
\Radj= 
\delta_\alpha \{\rho_\alpha^{-1}\}_{ji}  q^{n_i}+2\Delta
%+\delta_{\mu_i,0}c_{\Gamma,\rho,w}(\mu_j,0)_{ji}
+\sum_{\substack{0< k_j\\hk_j\in\Z+\mu_j}} q^{k_j}\sum_{c>0}S^\cusp_{n_i,k_j}(c,\rho)_{ji}\\
\cdot \frac{-2\pi i}{ch}\left(-\frac{k_j}{n_i}\right)^{\frac{w-1}{2}}
J_{1-w}\left(\frac{4 \pi i}{c} \sqrt{-k_jn_i}\right)
\end{multline*}
\begin{multline*}
\Ponj= \delta_\alpha \{\rho_\alpha^{-1}\}_{ji}  q^{-n_i}
+\sum_{\substack{0<k_j\\hk_j\in\Z+\mu_j'}} q^{k_j}\sum_{c>0}S^\cusp_{n_i,k_j}(c,\bar\rho)_{ji}\\
\cdot \frac{2\pi i^{w-2}}{ch}\left(-\frac{k_j}{n_i}\right)^{\frac{1-w}{2}}
J_{1-w}\left(\frac{4 \pi}{c}  \sqrt{-k_jn_i}\right) % CHECK PHASE IN BESSEL FUNCTION
\end{multline*}
\end{thm}
\begin{proof}
I forego the proof in the case that the cusp is not at $\infty$.  When the cusp is at $\infty$, the coefficients for the Rademacher sum follow from the proof of Theorem~\ref{convthm}.  The coefficients for the Poincar\'e series were computed in Theorem 3.2 of~\cite{Knopp:2004uq}.
\end{proof}

%The constant term in the Rademacher sum for $\mu_j=0$ is given in terms of the Kloosterman-Selberg zeta function $Kl_{\vec m,\vec n}(s,\rho) = \sum_c S(\vec n,\vec m,c,\rho)/c^{2s}$.
%\begin{equation*}
%c_{\Gamma,\rho,w}(\mu,0)_{ji} = \frac{1}{2h}(-2\pi i )^{2-w}\frac{1}{\Gamma(2-w)}Kl_{\mu_i-n,0}(1-w/2)_{ji}\left(\frac{\mu_i-n}{h}\right)^{1-w}
%\end{equation*}

This expression immediately gives the asymptotic behavior of the coefficients of the Rademacher sum, $\Radj$ using the known asymptotic behavior of the Bessel $J$ function.  For $z$ real, $J_\alpha(iz)\sim \frac{e^z}{\sqrt{2\pi z}}$, so the smallest $c$ term dominates for large $k$.  Let $c$ be the least positive such that $S^\cusp_{n_i,k_j}(c,\rho)_{ji}\neq 0$.  Then the $q^{k_j}$ coefficient in the Rademacher sum has asymptotic form

\begin{equation}\label{asymptotics}
 S^\cusp_{n_i,k_j}(c,\rho)_{ji}
\frac{-i}{\sqrt{2c}h}\frac{k_j^{\frac{2w-3}{4}}}{(-n_i)^{\frac{2w-1}{4}}}
\exp\left(\frac{4 \pi  \sqrt{-k_jn_i}}{c}\right).
\end{equation}

\section{Dimension of Spaces of Automorphic Integrals}\label{s3}
The dimension of the space of automorphic integrals can be computed explicitly in terms of dimensions of spaces of modular forms and the dimensions of the spaces of cusp forms.  Consider the space, $\mathcal J_w(m, \rho)$, of automorphic integrals that have a pole of order at most $m/h$ at $\infty$ and are holomorphic at cusps inequivalent to $\infty$.  There is a linear map, $\lambda$, that takes an automorphic integral $\vec f$ with shadow $\vec g$ and weight $w$ to $\lambda\vec f(\tau) = p(w,\tau;\vec g)$.  The kernel of $\lambda$ consists of modular functions with a single pole of order at most $m/h$ at $\infty$ by Equation~\ref{autodef}.   This is just the space $M^m_w(\rho)$.  The dimension of the image is bounded by the dimension of the space of cusp forms, $\mathcal S_{2-w}(\bar\rho)$.  Together, these give the bound,

%The following theorem of~\cite{Niebur:1974vw} remains valid in the vector case.
%\begin{thm}
%Let $G\in\mathcal S_{2-w}(\rho)$. Then $p(w,\gamma\infty;G) = 0$ for all $\gamma\in\Gamma$ if and only if $G$ is identically zero. That is, a unique cusp form is associated to each automorphic integral.
%\end{thm}
%Since the difference of two automorphic integrals with the same cusp form is a modular form, this leads to a bound on the dimension of the space of automorphic forms for sufficiently large $m$,
\begin{equation*}
\dim\mathcal J_w(m, \rho)\leq\dim\mathcal S_{2-w}(\bar\rho) + \dim M^m_w(\rho).
\end{equation*}

The dimensions of space of vector-valued automorphic forms are known.  They have been computed using the Riemann--Roch theorem for coherent sheaves~\cite{Freitag:2012wc} and previously using the Selberg trace formula~\cite{Skoruppa:1985vi}.  We will use the Riemann--Roch theorem in what follows.
\begin{thm}{(Riemann-Roch)}
Let $\mathcal M$ be a coherent sheaf on a compact Riemann surface $X$ and $\Omega$ be the canonical sheaf.  Then,
\begin{equation*}
\dim H^0(X,\mathcal M) - \dim H^0(X,\Hom_{\mathcal O_X}(\mathcal M,\Omega)) = \operatorname{deg}(\mathcal M) + \operatorname{Rank}(\mathcal M)(1-g)
\end{equation*}
\end{thm}

Let $X=(\bbH/\Gamma)\cup S$ where $S$ is the set of cusps of $\Gamma$.  Let $\mathcal M$ be a sheaf on $X$, such that for an open set $U\subset X$, $\mathcal M(U)$ consists of the local automorphic forms on $\tilde U\setminus S$ of weight $w$, multiplier system $\rho$ that are regular at the cusps, with the possible exception of the cusp at $\infty$, where it has a pole of degree at most $m$.  For sufficiently large $m$, %$\dim H^0(X,\Hom_{\mathcal O_X}(\mathcal M,\Omega))  = 0$ and 
Proposition 3.6 of~\cite{Freitag:2012wc} gives us that
\begin{align*}
\dim\mathcal{M}^m_w(\rho)  &= \dim H^0(X,\mathcal M) \leq  \operatorname{deg}(\mathcal M) + \operatorname{Rank}(\mathcal M)(1-g)\\
&\leq md+ d(1-g) +dw\left(g-1+\frac{|S|}{2} + \frac{1}{2}\sum_{b\in X_\Gamma\smallsetminus S}(1-\frac{1}{\e(\pi,b)})\right) - \sum_{x\in X_\Gamma}\sigma_\rho(x,w)
\end{align*}
for the elliptic points $X_\Gamma$, where $\e(\pi,b)$ is the order of the elliptic point $b$. $\sigma_\rho(x,r)$ is the sum of the characteristic numbers of the multiplier system at $x$, as defined in Lemma 3.1 of~\cite{Freitag:2012wc}.  $\sigma$ possess the property that $\sigma_\rho(x,w)+\sigma_{\bar\rho}(x,2-w)=d(1-1/\e(\pi,x))$, unless $x$ is a cusp. If $x$ is a cusp, then   $\sigma_\rho(x,r) = d-\dim V^{\rho(\Gamma_x)}$, where $V^{\rho(\Gamma_x)}$ is the fixed space of $\Complex^d$ that is preserved by $\rho(\gamma)$ for all $\gamma$ that preserve $x$.  Let $t_0$ be the dimension of the invariant subspace of $\rho(\Gamma)$. Similarly, 
\begin{align*}
\dim\mathcal{S}_{2-w}(\bar\rho) &= d(1-g) + d(2-w)\left(g-1+\frac{|S|}{2} + \frac{1}{2}\sum_{b\in X_\Gamma\smallsetminus S}(1-\frac{1}{\e(\pi,b)})\right) \\
 &\qquad {} - \sum_{x\in X_\Gamma}\sigma_{\bar\rho}(x,2-w)+t_0 \delta_{w,0} - \sum_{x\in S}\dim V^{\bar\rho(\Gamma_x)}
\end{align*}

Using these results, we see that for arbitrary $\Gamma$, rational weight $w$, multiplier system $\rho$ of dimension $d$, and for sufficiently large $m$, 
\begin{equation}\label{dim1}
\dim\mathcal J_w(m, \rho)\leq\dim\mathcal{M}^m_w(\rho) + \dim\mathcal{S}_{2-w}(\bar\rho) \leq md,
\end{equation}
unless $w=0$, in which case,
\begin{equation}\label{dim2}
\dim\mathcal J_w(m, \rho)\leq\dim\mathcal{M}^m_w(\rho) + \dim\mathcal{S}_{2-w}(\bar\rho) \leq md + t_0,
\end{equation}
%The dimensions of spaces of vertex valued modular functions over irreducible
%multiplier systems was computed by Bantay\cite{Bantay:2007vl} for $\SL$. In particular, let $\alpha$ (resp. $\beta_1$, $\beta_2$) be the 
where $t_0$ is the dimension of the invariant subspace of $\rho(\Gamma)$.

We will see that these bounds are exact.  Consider the set of Rademacher sums, $\Rad$ as $i$ ranges from $1$ to $d$ and as $n$ ranges from $(\mu_i-1)/h$ to $(\mu_i-m)/h$.  Observe from the $q$-series expansion that these Rademacher sums have pole of order $n$ in the  $i$th component.   Since those are different poles for each of the $i$ and $n$, the Rademacher sums must be linearly independent.  There are exactly $md$ of these sums, all of which lie in $\mathcal J_w(m, \rho)$.

When $w=0$, then $w_0$ is the dimension of the invariant subspace under $\rho(\Gamma)$.  The invariant vectors can be taken to be constant vector-valued functions over $\bbH$.   Because these functions are invariant under $\rho(\Gamma)$ and $w=0$, these constant functions are modular forms under $\Gamma$, and thus are automorphic integrals in $\mathcal J(m,\rho)$.  Because they do not diverge at the cusps, these functions must be linearly independent from the Rademacher sums.

Together with equations~\ref{dim1} and~\ref{dim2}, we see that the inequality is exact and obtain the following theorem.

\begin{thm}\label{basis}
The space of automorphic forms, $\mathcal J(m,\rho)$ is spanned by $\Rad$
 and $0\leq i\leq d$ and $m<n<0$ with $hn\in\Z+\mu_i$.  If $w=0$, the space of automorphic forms is spanned by $\Rad$ along with the constant modular forms defined above.
\end{thm}

The theorem extends to the space of weakly holomorphic modular functions where poles of order at most $m$ can exist at any of the cusps.  In that case, the bases consist of Rademacher sums evaluated with poles at any of the cusps along with any constant functions.

\section*{Acknowledgements}
I am grateful to John Duncan, Miranda Cheng, Francesca Ferrari, Natalie Paquette, and Shamit Kachru for useful discussion and invaluable assistance preparing this manuscript.

%
% APPENDIX KLOOSTERMAN
%

\appendix
\section{Asymptotics of Kloosterman-Selberg Zeta Functions}\label{kloosterman}

The asymptotic behavior and convergence of the Kloosterman-Selberg zeta functions will be necessary to bound the terms in the Rademacher sum.  These limits will be used in the convergence proof in Appendix~\ref{convergence}.

In this appendix, we will sacrifice generality for the sake of ease of presentation and will assume that the Rademacher sums and the Kloosterman sums are generated by the cusp at $\infty$.  The $\infty$ will be suppressed in the notation.

Fix a finite volume Fuchsian group $\Gamma\in\GL$ with cusp at $\infty$ and normal representation $\rho$.  Let $h$ be the width of the cusp at $\infty$, and $\rho(T^h) = [e(\vec\mu)]$ for $0\leq\mu_i<1$.   Recall the Kloosterman sum from Definition~\ref{Sdef}.  We can similarly define the matrix valued Kloosterman-Selberg zeta function.
\begin{defin}
For $c'>0$, and $hn\in\Z+\mu_j,  hm\in\Z+\mu_i$, the coefficients of the matrix-valued Kloosterman-Selberg zeta function are given in terms of the components of the Kloosterman-Selberg sum by
\begin{equation*}
\Kl_{m_i,n_j}(s,\rho)_{ji,<K} = \sum_{0<c<K}\frac{S_{m_i,n_j}(c,\rho)_{ji}}{c^{2s}}
\end{equation*}
and $\Kl_{m_i,n_j}(s,\rho) = \lim_{K\to\infty}\Kl_{m_i,n_j}(s,\rho)_{ij,<K} $ when the limit converges.
\end{defin}

There are trivial bounds on the sums and zeta functions that arise from the finite volume of $\Gamma$~\cite{Iwaniec:2002vu}.

%Trivial bounds
\begin{thm}\label{trivial}
For $\Gamma$ a Fuchsian group with finite volume, and $\rho$ unitary, the Kloosterman sums satisfy ${S_{m_i,n_j}(c,\rho)_{j,i}= O\left(c^2\right)}$ and $\Kl_{m_i,n_j}\left(\frac{1}{2}\right)_{ji,<K} = O(K)$.  The expression $\Kl_{m_i,n_j}(s,\rho)$ converges absolutely for $\re s>1$.
\end{thm}

For certain parts of the convergence proof, we will need better bounds.  They are presented here, but because the proofs are complex and the details are not relevant to the behavior of the Rademacher sums, we delay the proofs until Appendices~\ref{proofoflemma} and~\ref{otherproof}.  In presenting these proofs, we use the symbol $Z_{m_i,n_j}(s,\rho)_{ij}$ for the meromorphic continuation in $s$ of $\Kl_{m_i,n_j}(s,\rho)_{ji}$.

\begin{thm}
\label{Kllimit}
Let $m_in_j\neq 0$,  $0<\epsilon\leq \frac{1}{2}$, $\frac{1}{2}<\re(s)\leq 1+\epsilon$, and $|\im(s)|<1$.  Then,
\begin{equation*}
|Z_{m_i,n_j}(s,\rho)_{ji}| = O\left(|m_in_j|^{1-\re(s)+\epsilon}\frac{|\im(s)|^{1-\re(s)+\epsilon}}{\re(s)-1/2}\right)
\end{equation*}
as $\im(s)\to\infty$. The matrix-valued constant depends only on $\Gamma$ and $\rho$.  Furthermore, over this range, $Z_{m_i,n_j}(s,\rho)_{ji}$ is holomorphic  except at finitely many simple poles that lie on the real line. 
\end{thm}
\begin{proof}
A proof is provided in Appendix~\ref{proofoflemma}.
\end{proof}
The part of the $s$ strip that is relevant is a small vertical band containing 1, but no other poles.  We can find such a band on which $Z$ satisfies slightly stricter limiting behavior.   
\begin{cor}\label{cor0}
For some $0<a<\frac{1}{2}$, $Z_{m_i,n_j}(s,\rho) $ has no poles in the interval $(1-a,1)$ and has the following limiting behavior: if $1-a<\re(s)<1+a$ and $|Im(s)|>0$, then \begin{equation*}
|Z_{m_i,n_j}(s,\rho)_{ji}| = O\left(|m_in_j|\frac{|\im(s)|^{1/2}}{\re(s)-1/2}\right)\end{equation*}
as $\im(s)\to\infty$.  The matrix-valued constant depends only on $\Gamma$ and $\rho$.
\end{cor}

In the case that $m$ or $n$ is zero, the limiting bound needs to be treated separately.  The following bound holds.

\begin{thm}\label{cor00}
For $m_in_j=0$ and one of $m_i,n_j$ nonzero, the following bound holds uniformly for some $a$ over the range $1-a<\re(s)<1+a$,
\begin{equation*}
|Z_{m_i,n_j}(s,\rho)_{ji}| = O\left({|\im(s)|}^{1/2}\right).
\end{equation*}

The matrix-valued constant may depend only on $\Gamma$, $\rho$, $m_i$, and $n_j$. 
\end{thm}  
\begin{proof}
It suffices to show that this is true for $m_i=0$, where it is proven in Appendix~\ref{otherproof}.
\end{proof}

\begin{cor}\label{cor}
%The sum $\Kl_{m,n}(1)_{<K} = O(n)$.  This expression converges uniformly in both $K$ and $n$.
For fixed $m$, the sum $\Kl_{m_i,n_j}(1)_{ji}= O(n)$ converges uniformly in $n$.
\end{cor}
\begin{proof}
With Corollary~\ref{cor0} and Theorem~\ref{cor00}, the proof is component-wise the same as in the scalar case.  See~\cite{Duncan:2011uj}, Sections 3.1 through 3.3 for details.
\end{proof}

%
% APPENDIX A
%

\section{Convergence of the Rademacher Sum}
\label{convergence}

To prove convergence of the vector-valued Rademacher sums, we follow the approach of Duncan and Frenkel~\cite{Duncan:2011uj}.  Some additional changes are needed to account for nontrivial multiplier systems and the vector-valued nature of the sum.

\begin{thm}\label{convthm}
Fix a finite volume Fuchsian group $\Gamma\in\GL$ with cusp at $\infty$ and normal representation $\rho$ of weight $w$.  For $\re(w)\leq0$, and $0>nh\in\Z+\mu_i$,
\begin{equation*}
\Rad =\vec\Delta+ \lim_{K\to\infty}\sum_{\gamma\in\GG_{K,K^2}}j_w(\gamma,\tau)^{-1}\rho(\gamma)^{-1}\rad^{n_i}_w(\gamma,\tau)q_\gamma^{n_i}
\end{equation*}
converges locally uniformly to a holomorphic function on $\mathbb{H}$.  The Fourier coefficients of this limit are as described in Theorem~\ref{fouriercoefficients}.
\end{thm}
\begin{proof}

Consider the expression 

\begin{equation*}
\Rad = \vec\Delta+\lim_{K\to\infty}\sum_{\gamma\in\GG_{K,K^2}}j_w(\gamma,\tau)^{-1}\rho(\gamma)^{-1}\rad^{n_i}_w(\gamma,\tau)q_\gamma^{n_i}.
\end{equation*}
We can substitute in the sum expression in equation~\ref{eq545} for the regularization factor and set $\gamma\tau-\gamma\infty=-1/(c(c\tau+d))$.  We change the sum over $\GGstar$ to a double coset sum by summing over $\gamma T^{hl}$ for $\gamma\in\GGstarG$ and $l\in\Z$.  We express the sum in terms of bounds on $c$ and $d$ with the implicit assumption that the sum is restricted to $c,d$ such that $\sm**cd\in\Gamma$.. After including the $K$ limits the $j$th component of the Rademacher sum is
\begin{multline*}
\Radnovec= \delta_{ij}q^{n_i}+\Delta_j+\lim_{K\to\infty}\sum_{0<c<K}\sum_{0<d<ch}\sum_{\substack{|l|<K^2/ch\\l\in\Z}}
e(-l\mu_j)\rho(\gamma)^{-1}_{ji}e\left(n_i\gamma T^{hl}\tau\right)\\
{}\cdot\e\left(n_i(-\gamma T^{hl}\tau+\gamma T^l\infty)\right)\sum^{\infty}_{m=0}\frac{(-2\pi in_ih/c)^{m+1-w} (c\tau+d+lch)^{-m-1}}{\Gamma(m+2-w)}
\end{multline*}
plus a correction term that is $O(1/K^2)$, arising from the reordering.  We have used the fact that for $T^h\tau = \tau+h$, $\rho(T) = [e(\vec\mu)]$ and $n_ih\in\Z+\mu_i$.  The second two sums converge absolutely, so we can change the order of summation in the inner sum.
\begin{multline*}
\Radnovec= \delta_{ij}q^{n_i}+\Delta_j+\lim_{K\to\infty}\sum_{0<c<K}\sum_{0<d<ch}\rho(\gamma)^{-1}_{ji}e\left(n_i\gamma\infty\right)\\
\cdot\sum_{m=0}^\infty\frac{(-2\pi in_i/c)^{m+1-w}}{\Gamma(m+2-w)}
\sum_{\substack{|l|<K^2/ch\\l\in\Z}}e(-l\mu_j)(c\tau+d+lch)^{-m-1}.
\end{multline*}

This sum can be simplified using the Lipschitz summation formula\cite{Knopp:2001ks}.
\begin{prop}(Lipshitz summation formula) For $\im(w)>0, p\geq1,$ and $0\leq\alpha<1$,
\begin{multline*}
\sum_{\substack{|n|<N\\n\in\Z}}
\frac{e(-n\alpha)}{(w+n)^{p}}
-\frac{(-2\pi i)^p}{\Gamma(p)}
\sum_{\substack{m+\alpha>0\\m\in\Z}}(m+\alpha)^{p-1}e((m+\alpha)w)%\\
= \begin{cases} -\pi i+O(1/N) &\mbox{if } \alpha =0, p=1\\ 
O(1/N^2) & \mbox{otherwise} \end{cases}
\end{multline*}
as $N$ approaches $\infty$.  This convergence is locally uniform in $\tau$.
\end{prop}

We continue neglecting $O(1/K^2)$ terms so that the result of the Lipschitz summation is
\begin{multline*}
\Radnovec= \delta_{ij}q^{n_i}+\Delta_j+\lim_{K\to\infty}
\sum_{0<c<K}
\sum_{0<d<ch}\rho(\gamma)^{-1}_{ji}e\left(n_i\gamma\infty\right)
\sum_{m=0}^\infty\frac{(-2\pi in_i/c)^{m+1-w}}{\Gamma(m+2-w)}\\
\cdot\left(-\frac{\pi i}{ch}\delta_{m0}\delta_{\mu_j0}+
\sum_{\substack{k'\in\Z\\k'+\mu_j>0}} \frac{1}{\Gamma(m+1)}\left(\frac{-2\pi i}{ch}\right)^{m+1} (k'+\mu_j)^m e\left(\frac{d(k'+\mu_j)}{ch}\right)q^{(k'+\mu_j)/h}\right)
\end{multline*}

The sum over $m$ is absolutely convergent, so we can rearrange the sum.  When we do so, the finite sum over $d$ becomes a Kloosterman sum.  Writing $h k_j = k+\mu_j$, 
\begin{multline}\label{breakpoint}
\Radnovec= \delta_{ij}q^{n_i}+\Delta_j+\lim_{K\to\infty}\sum_{0<c<K}
\sum_{m=0}^\infty\frac{(-2\pi in_i/c)^{m+1-w}}{\Gamma(m+2-w)}\\
\cdot\left(-\frac{\pi i}{ch}\delta_{m0}\delta_{\mu_j0}S_{n_i,0}(c,\rho)_{ji}+
\sum_{\substack{hk_j\in\Z+\mu_j\\k_j>0}} \frac{1}{\Gamma(m+1)}\frac{1}{h}\left(\frac{-2\pi i}{c}\right)^{m+1} k_j^mS_{n_i,k_j}(c,\rho)_{ji}q^{k_j}
\right)
\end{multline}

To make the sum more tractable, we split up the individual components.  Write
\begin{equation*}
%\left[\Rad\right]_j = \delta_{ij}q^{-n_i/h} + \delta_{\mu_j0}c_{\Gamma,\rho,w}(\mu_j,0)_{ji} +R_{+,j} + R_{0,j}.
\Radnovec= \delta_{ij}q^{n_i} + 2\Delta_j +R_{+,j} + R_{0,j}.
\end{equation*}
One of the $\Delta_j$ terms arises from $-\pi i/(ch)\delta_{m0}\delta_{\mu_j0}S_{0,n_i}(c,\rho)_{ji}$ in equation~\ref{breakpoint}.  The $R_{+,j}$ consists of the sum across all of $c$ and across positive $m$.  The $R_{0,j}$ term consists of the sum across $c$ for $m=0$.  We will see that the rearrangement of the sum is valid because $R_{+,j}$ converges absolutely.

Consider the $R_{+,j}$ part of the sum.
\begin{multline*}
R_{+,j}=
\lim_{K\to\infty}\sum_{0<c<K}
\sum_{m=1}^\infty\frac{(-2\pi in_i/c)^{m+1-w}}{\Gamma(m+2-w)}%\\
%\cdot
\sum_{\substack{hk_j\in\Z+\mu_j\\k_j>0}} \frac{1}{\Gamma(m+1)}\frac{1}{h}\left(\frac{-2\pi i}{c}\right)^{m+1} k_j^mS_{n_i,k_j}(c,\rho)_{ji}q^{k_j}
\end{multline*}

%For $\re s>1$, the Kloosterman sum is $O(c^2)$ by theorem~\ref{trivial}.  
For $\re s>1$, the Kloosterman-Selberg zeta function converges absolutely and has limiting behavior given by Theorem~\ref{Kllimit}.  We can therefore rearrange the integral and use the definition of the Kloosterman-Selberg zeta function to see the following expression for $R_+$,
\begin{multline*}
R_{+,j}=
\sum_{m=1}^\infty
\frac{(-2\pi in_i)^{m+1-w}}{\Gamma(m+2-w)}
\sum_{\substack{hk_j\in\Z+\mu_j\\k_j>0}} 
\frac{1}{\Gamma(m+1)}\frac{1}{h}\left(\frac{-2\pi i}{c}\right)^{m+1} k_j^mq^{k_j}
Z_{n_i,k_j}(m+1-w/2,\rho)_{ji},
\end{multline*}
which is absolutely and locally uniformly convergent by the trivial bounds in Theorem~\ref{trivial}.

Now consider the $R_{0,j}$ term.
\begin{align*}
R_{0,j} &=
\lim_{K\to\infty}
\sum_{0<c<K}
\frac{(-2\pi in_i/c)^{1-w}}{\Gamma(2-w)}%\\
%\cdot
\sum_{\substack{hk_j\in\Z+\mu_j\\k_j>0}} \frac{-2\pi i}{ch} S_{n_i,k_j}(c,\rho)_{ji}q^{k_j}
\\
&=\lim_{K\to\infty}
\frac{(-2\pi in_i)^{1-w}}{\Gamma(2-w)}\frac{-2\pi i}{h}
\sum_{k_j\in\Z+\mu_j} Kl_{n_i,k_j}(1-w/2,\rho)_{ji,<K} q^{k_j},
\end{align*}
which converges absolutely and locally uniformly for $w<0$ by the trivial bounds in Theorem~\ref{trivial}.  In the case that $w=0$, this converges locally uniformly by the bounds in Corollary~\ref{cor}.

For the Fourier coefficients, return to equation~\ref{breakpoint}. Using the convergence results above, the inner sum can be rearranged to allow us to use the infinite series expression for the Bessel $ J$ function, which gives
%\begin{multline*}
%\left[\Rad\right]_j= \delta_{i,j}q^{(\mu_i-n)/h}+\sum_{k_j\in\Z+\mu_j} q^{(k+\mu_j)/h} \lim_{K\to\infty}\sum_{0<c<K}S_{\mu_i-n,k+\mu_j}(c,\rho)_{ji}\\
%\cdot \frac{-2\pi i}{ch}\left(\frac{k+\mu_j}{n-\mu_i}\right)^{\frac{w-1}{2}}
%J_{1-w}\left(\frac{4 i \pi  \sqrt{(k+\mu_j) (n-\mu_i)}}{c h}\right)
%\end{multline*}
\begin{multline*}
\Radnovec= \delta_{ij}q^{n_i}+2\Delta_j+\sum_{\substack{hk_j\in\Z+\mu_j\\k_j>0}} q^{k_j}\sum_{c>0}S_{n_i,k_j}(c,\rho)_{ji}\\
\cdot \frac{-2\pi i}{ch}\left(-\frac{k_j}{n_i}\right)^{\frac{w-1}{2}}
J_{1-w}\left(\frac{4 \pi i}{c}  \sqrt{-k_jn_i}\right)
\end{multline*}

\end{proof}

%
% APPENDIX B
%

\section{Limiting Behavior of the Kloosterman-Selberg Zeta Function}\label{proofoflemma}

Understanding the behavior of $\Kl_{m,n}(1)_{<K}$ requires a generalization of a result by Goldfeld-Sarnak, which we will elaborate on below \cite{Goldfeld:1983wq,Pribitkin:2000jl}.  Recall that $Z_{m_i,n_j}(s,\rho)$ is the analytic continuation of $Kl_{m,n}(s,\rho)$.  Then, we have the following bound

\begin{thm*}
Let $m_i n_j\neq 0$,  $0<\epsilon\leq \frac{1}{2}$, $\frac{1}{2}<\re(s)\leq 1+\epsilon$, and $|\im(s)|<1$.  Then,
\begin{equation*}
|Z_{m_i,n_j}(s,\rho)_{ji}| = O\left(|m_in_j|^{1-\re(s)+\epsilon}\frac{|\im(s)|^{1-\re(s)+\epsilon}}{\re(s)-1/2}\right)
\end{equation*}
as $\operatorname{Im}(s)\to\infty$. The matrix-valued constant may depend on $m, n,$ $\Gamma$ and $\rho$.  Furthermore, over this range, $Z_{m_i,n_j}(s,\rho)_{ji}$ is holomorphic  except at finitely many simple poles that lie on the real line. 
\end{thm*}
%\begin{thm*}
%Let $mn\neq 0$, and  $\frac{1}{2}<\re(s)\leq 1$.  Then,
%\begin{equation*}
%Z_{m_i,n_j}(s,\rho) = O\left(\frac{|s|^{1/2}}{\operatorname{Re}(s)-1/2}\right)
%\end{equation*}
%as $\operatorname{Im}(s)\to\infty$. The constant may depend on $m, n,$ $\Gamma$ and $\rho$.  Furthermore, over this range, $Z_{m_i,n_j}(s,\rho)$ is holomorphic  %except at finitely many simple poles that lie on the real line. 
%\end{thm*}

The theorem will follow trivially from the two lemmas below.

Following %Selberg~\cite{Selberg:1965vr} and 
 Pribitkin \cite{Pribitkin:2000jl}, let $z=x+iy$ for $x,y\in\bbR$ and define the space, $L^2(\Gamma\backslash\mathcal H,\rho,k)$ to be the Hilbert space of functions $f:\mathcal H\to\Complex^{d\times d}$ satisfying
\begin{equation*}
f(\gamma z) = \rho(\gamma) e^{i k \arg(cz +d)} f(z)
\end{equation*}
and that are square-integrable:
\begin{equation*}
\intint_{\Gamma\backslash\mathcal H} |f_{ij}(z)|^2 \frac{dxdy}{y^2}<\infty
\end{equation*}
observe that that condition does not depend on the choice of fundamental domain because it is true if and only if
$\intint_{\Gamma\backslash\mathcal H} \sum_i|f_{ij}(z)|^2 \frac{dxdy}{y^2}<\infty$
for all $j$.  The later expression is independent of choice of the $\Gamma\backslash\mathcal H$ by the first condition and because the $\rho$ are unitary.

Consider the hyperbolic Laplacian of weight $k$,
\begin{equation*}
\Delta_k = y^2\left(\frac{\partial^2}{\partial x^2}+\frac{\partial^2}{\partial y^2}\right) -iky\frac{\partial}{\partial x}
\end{equation*}
which acts component-wise on functions $f$.  $\Delta_k$ can be extended to $L^2(\Gamma\backslash\mathcal H,\rho,k)$ as an self-adjoint operator.\cite{Pribitkin:2000jl}.   It is a standard result that there are only finitely many eigenvalues of $\Delta_k$ in the range $0<\re(k)<\frac{1}{2}$, which means that the resolvent operator ${(\Delta_k+s(1-s))^{-1}}$ is holomorphic over $\frac{1}{2}<\re(s)<1$ except at finitely many points that lie on the real line.\cite{Pribitkin:2000jl}%\cite[ch.~4]{Iwaniec:2002vu}.

 In order to approximate the size of the Kloosterman sum, we consider the vector-valued generalization of Selberg's non-holomorphic Poincar\'e series~\cite{Selberg:1965vr}.  Let $\hat i^T$ be the transpose of the $i$th unit vector. Given $m_i>0$ with $hm_i\in\Z+\mu_i$,
\begin{equation*}
P_{m_i}(z,s,\rho,k, i) = \sum_{\gamma\in\GG}{\hat i}^T\e(m_i\gamma z)\bar\rho(\gamma) \frac{y^s}{|cz+d|^{2s}}e^{-i k \arg(cz +d)}
\end{equation*}
converges absolutely and uniformly for $\re(s)>1$ and
is in $L^2(\Gamma\backslash\mathcal H,\rho,k)$.  The vector-valued $P$ satisfies the following functional equation
\begin{equation}
\label{recursion}
P_{m_i}(z,s,\rho,k) = -4\pi\left(s-\frac{k}{2}\right)(\Delta_k+s(1-s))^{-1}P_{m_i}(z,s+1,\rho)m_i.
\end{equation}
Because of the functional equation and because $(\Delta_k+s(1-s))^{-1}$ is a holomorphic operator in the sense of~\cite{Selberg:1965vr}, $P_{m_i}(z,s,\rho,k)$ can be meromorphically continued  continued to $\frac{1}{2}<\re(s)<1$, where it is holomorphic except for finitely many simple poles that lie on the real line.

%We know what these eigenvalues can be.  Let $\Gamma'\subset\Gamma$ be the subgroup on which $\rho$ is diagonal.  Then the $i$th diagonal element $\rho_i$ of $\rho$ is a scalar multiplier system of $\Gamma'$.  It follows that $P_{m_i}(z,s,\rho,k)_{ij}\in L^2(\Gamma\backslash\mathcal H,\rho_i,k)$.  Since $\rho_i$ is a scalar multiplier system, we can use Weil's results~\cite{Weil:1948vc} to observe that $P_{m_i}(z,s,\rho,k)_{ij}$ has contributions from finitely many eigenfunctions, all of which have eigenvalues between $\frac{1}{2}$ and $1$ exclusive.

%The equivalent of Lemma 1 in~\cite{Goldfeld:1983wq} follows trivially from the above.

The spectral theory of $\Delta_k$ will allow us to give bounds on $P_{m_i}(z,s,\rho,i)$.  Together with Lemma~\ref{lemmatwo}, this will give us bounds on $Z_{m_i,n_j}(s,\rho)$.

% BEGIN LEMMA 1

\begin{lem}\label{lemmaone}
For $\frac{1}{2}<\re(s)\leq\frac{3}{2}$, and $|\im(s)|>1$, 
\begin{equation*}
\intint_{\Gamma\backslash\mathcal H}|P_{m_i}(z,s,\rho,k,i)|^2\frac{dxdy}{y^2} = O\left(\frac{m_i^2}{(\re(s)-1/2)^2}\right),
\end{equation*}
where the constant depends on $\Gamma,\rho$, and $k$.
\end{lem}
\begin{proof}

There is a bound for $\frac{3}{2}<\re(s)<\frac{5}{2}$,
\[\intint_{\Gamma\backslash\mathcal H}|P_{m_i}(z,s,\rho,k,i)|^2\frac{dxdy}{y^2} = O(1)\]
as $m_i$ varies, where the constant depends on $\Gamma$ and $\rho$.
%For $\frac{3}{2}\leq\re(s)\leq\frac{5}{2}$, $|P_{m_i}(z,s,\rho,k,i)|$ converges absolutely and uniformly. 
For $\Gamma\not\subset\SL$, this claim is nontrivial, but the vector valued case and the scalar case follow from the bounds in Section 2.6 of Iweniac\cite{Iwaniec:2002vu}.  

Now we can use this bound and the functional equation to find a bound for $\frac{1}{2}<s<\frac{3}{2}$.  We know from the spectral theory a bound on the magnitude of the resolvent:
\begin{equation*}
|\Delta_k^{-1}| \leq \frac{1}{\operatorname{distance}(k,\operatorname{spectrum}(\Delta_k))}.
\end{equation*}
The distance is known since $\Delta_k$ is self-adjoint, so its spectrum is real.  When applied to the recursion formula,~\ref{recursion}, this gives $|\Delta_k^{-1}| \leq1/(|\im(s(1-s))|)$ and for $\frac{1}{2}<\re(s)\leq\frac{3}{2}$,
\begin{equation*}
\left(\intint_{\Gamma\backslash\mathcal H}|P_{m_i}(z,s,\rho,k,i)|^2\frac{dxdy}{y^2}\right)^\frac{1}{2} = O\left(m_i\frac{|s-k/2|}{|\im(s)|\;|2\re(s)-1|}\right),
\end{equation*} 
whereupon the bound $|\im(s)|>1$ gives us the desired result.
\end{proof}

% END LEMMA 1

We can approximate the size of the Kloosterman-Selberg zeta function directly using the non-holomorphic Poincar\'e series and use Lemma~\ref{lemmaone} to obtain a bound on the zeta function.

%BEGIN LEMMA 2

\begin{lem}\label{lemmatwo}
For $m_i,n_j>0$ and $\re(s)>\frac{1}{2}$,
\begin{multline*}
\intint_{\Gamma\backslash\mathcal H}
\overline{P_{n_j}(z,\bar s+2,\rho,k, j)}
P_{m_i}(z,s,\rho,k,i)^T
\frac{dxdy}{y^2} = \delta_{m_i,n_j}\delta_{i,j} h (4\pi n_j)^{-2s-1}\Gamma(2s+1)
\\+ (-i)^k4^{-s-1}\pi^{-1}
n_i^{-2}
\frac{\Gamma(2s+1)}{\Gamma(s+k/2)\Gamma(s-k/2+2)}Z_{m_i,n_j}(s,\rho)_{ji} + R(s)_{ji},
\end{multline*}
where $R(s)_{ji}=O\left(\left|\operatorname{Re}(s)-\frac{1}{2}\right|^{-1}\right)$ is holomorphic over this region.
\end{lem}
\begin{proof}

We work directly from the definition of the non-holomorphic Poincar\'e series.  With a few changes of variables and exploiting one of the $\Gamma$ sums to extend the integral from the fundamental domain to all of $\mathcal H$, we can obtain,
\begin{multline*}
\intint_{\Gamma\backslash\mathcal H}P_{n_j}(z,\omega,\rho,j)^{-1} P_{m_i}(z,s,\rho,i)\frac{dxdy}{y^2}=
\delta_{m_i,n_j}\delta_{i,j}hn_j^{-2s-1}\Gamma(2s+1)+\\
\sum_{c\in\Z^+} \intint_\mathcal{H}
\e\left(-n_j(xy-iy)\right)
\frac{S_{m_i,n_j}(c,\rho)_{ji}}{|c|^{2s}}\times\\
\e\left(\frac{-m_i}{yc^2(x+i)}\right)\frac{y}{(x^2+1)^s}\left(\frac{x+i}{(x^2+1)^{1/2}}\right)^{-k}dxdy.
\end{multline*}

The expression can be simplified through the use of integral identities for the Whittaker function, $W$.  For $Y>0,$ and 
$\re(\alpha+\beta) >1$,~\cite{Maass:1983ug}
\begin{equation*}
\int^\infty_{-\infty} \frac{e(-Yx)}{(x+i)^\alpha(x-i)^\beta}dx = \frac{\pi(-i)^{\alpha-\beta}(\pi Y)^{(\alpha+\beta)/2-1}}{\Gamma(\alpha)}W_{(\alpha-\beta)/2,(\alpha+\beta-1)/2}(4\pi Y).
\end{equation*}
We also have the Mellin transform for $N>0$ and $\re(s+1/2\pm\mu)>0$,~\cite{Jeffrey:2007uy}
\begin{equation*}
\int^\infty_0 e(-Ny)y^{s-1}W_{\beta,\mu}(4\pi N y)dy = (4\pi N)^{-s}\frac{\Gamma(s+1/2+\mu)\Gamma(s+1/2-\mu)}{\Gamma(s-\beta+1)}.
\end{equation*}

Using these transformations, we can simplify to give% by taking a copy of the identity matrix out of $\left[\e\left(-\frac{n-\vec\mu}{q}(xy-iy)\right)\right]$ and integrating the rest. Setting $\bar\omega = s+2$ to eliminate the first $\Gamma$ functions then gives us
\begin{multline*}
\intint_{\Gamma\backslash\mathcal H}
\overline{P_{n_j}(z,\omega,\rho,j)}
P_{m_i}(z,s,\rho,i)^T\frac{dxdy}{y^2}=
\delta_{m_i,n_j}\delta_{i,j} h (4\pi n_j)^{-2s-1}\Gamma(2s+1)
\\+(-i)^k4^{-s-1}\pi^{-1}n_i^{-2}\frac{\Gamma(2s+1)}{\Gamma(s+k/2)\Gamma(s-k/2+2)}Z_{m_i,n_j}(s,\rho)_{ji} + R(s)_{ji}.
\end{multline*}
where
\begin{multline*}
R(s)_{ji}=\sum_{c>0}\intint_\mathcal{H}
\e\left(-n_j(xy-iy)\right)
\frac{S(m,n,c,\rho)}{|c|^{2s}}\times\\
\left(-1+\e\left(\frac{-m_i}{yc^2(x+i)}\right)\right)
\frac{y^2}{(x^2+1)^s}\left(\frac{x+i}{(x^2+1)^{1/2}}\right)^{-k}\frac{dxdy}{y}.
\end{multline*}

Using the fact that 
\begin{multline*}
\intint_\mathcal{H}\e\left(-n_j(xy-iy)\right)
\left(-1+\e\left(\frac{-m_i}{yc^2(x+i)}\right)\right)\times\\
\frac{y}{(x^2+1)^s}\left(\frac{x+i}{(x^2+1)^{1/2}}\right)^{-k}dxdy\ll\frac{c^{-2}}{\left|\operatorname{Re}(s)-\frac{1}{2}\right|},
\end{multline*}

we see that 

\begin{equation*}
|R(s)_{ji}| \leq\frac{1}{\left|\operatorname{Re}(s)-\frac{1}{2}\right|}\sum_{c>0} \frac{|S_{ m_i,n_j}(c,\rho)_{ji}|}{|c|^{2\re(s)+2}}.
\end{equation*}
componentwise.  The last part is the sum generating $Kl_{m_i,n_j}(\re(s)+1,\rho)_{ji}$, which converges absolutely uniformly over the range $\re(s)>\frac{1}{2}$.  Furthermore, the partial absolute sums are $O(1)$, where the constant depends on $\re(s)$ and $\Gamma$ \cite{Pribitkin:2000jl}.  Thus, $|R(s)|=O\left(\left|\operatorname{Re}(s)-\frac{1}{2}\right|^{-1}\right)$ and holomorphic over the region $\re(s)>\frac{1}{2}$.

$R(s)$ is holomorphic for $\operatorname{Re}(s)>\frac{1}{2}$ and $R(s)\in O\left(\frac{1}{\operatorname{Re}(s)-\frac{1}{2}}\right)$.
\end{proof}

% END LEMMA 2

The theorem follows directly from Stirling's approximation and these lemmas.

%
% APPENDIX C
%

\section{Limiting Behavior of Eisenstein Series}\label{otherproof}

To obtain bounds on the Kloosterman-Selberg zeta function, we express the Fourier coefficients of certain vector-valued Eisenstein series in terms of the zeta functions.  Decomposing the vector-valued Eisenstein series into the sum of scalar-valued Eisenstein series evaluated at specific cusps will allow us to apply the standard theory of Eisenstein series to obtain bounds on the Fourier coefficients.  Background on the behavior of Eisenstein series can be found in~\cite{Bruggeman:1994uha} and~\cite{Kubota:1973vda}.

We consider only the weight 0 case, because for higher weights the trivial estimate suffices for the convergence result in Appendix~\ref{kloosterman}.  Consider a Fuchsian group, $\Gamma$ with finite volume and a cusp at infinity and a normal multiplier system $\rho$ of weight $0$ such that its $\mu_i=0$ .We can define the vector-valued non-analytic Eisenstein series of $\Gamma$ and $\rho$ analogously to the Poincar\'e series,
%\begin{equation*}
%E(z,s,\rho,k,i) = \sum_{\gamma\in\GG}\rho(\gamma)^{-1} \frac{y^s}{|cz+d|^{2s}} e^{-i\arg j_k(\gamma,z)}\hat i
%\end{equation*}
\begin{equation*}
\vec E(z,s,\rho,0,i) = \sum_{\gamma\in\GG}\rho(\gamma)^{-1} \frac{y^s}{|cz+d|^{2s}}\hat i
\end{equation*}

Now let $\Gamma_i\in\Gamma$ be the preimage under $\rho$ of the diagonal matrices in $\GL$ such that the $i$th coefficient along the diagonal is 1.  Since $\rho$ is normal with $\alpha_i=0$, we know that $\Gamma_i\supset\Ginf$ and that $\Gamma_i\backslash\Gamma$ is finite.  Fix a set of coset representatives $\beta_r$ such that $\bigcup_r \Gamma_i\alpha_r = \Gamma$.  We can separate out the dependence on the non-trivial parts of the multiplier system.

%\begin{align*}
%E(z,s,\rho,k,i)& = \sum_r \rho(\beta_r)^{-1} e^{-i \arg(j_k(\alpha_r,\tau))}\sum_{\gamma\in\GG_D}\rho(\gamma)^{-1} \frac{\im(\beta_rz)^s}{|c\alpha_rz+d|^{2s}} e^{-ik\arg(c\beta_rz+d)}\hat i\\
%&=\sum_r \rho(\beta_r)^{-1}\hat i e^{-i \arg(j_k(\beta_r,\tau))} E_{\Gamma_D}(\beta_rz,s,\rho_{ii},k,i).
%\end{align*}
\begin{align*}
\vec E(z,s,\rho,0,i)& = \sum_r \rho(\beta_r)^{-1} \sum_{\gamma\in\beta_r\Gamma_\infty\beta_r^{-1}\backslash\Gamma_i}
\rho(\gamma)^{-1} \frac{\im(\beta_rz)^s}{|c\alpha_rz+d|^{2s}}\hat i\\
&=\sum_r \rho(\beta_r)^{-1}\hat i\cdot E_{\Gamma_i}(\beta_rz,s,1,0).
\end{align*}

We are using the scalar Eisenstein sum  $E_{\Gamma_i}(\beta_rz,s,1,0)$ over ${\Gamma_i} $ as opposed to over $\Gamma$ and evaluated at a different cusp.  Through standard techniques,~\cite{Bruggeman:1994uha} we know that each of the $E_{\Gamma_D}$ permits an analytic continuation in $s$ to some $1-a<\re(s)<1+a$ except for a pole at $1$.  Moreover, we know that the Fourier coefficients of the scalar Eisenstein series in the sum are bounded over the strip as $\im(s)\to\infty$\cite{Duncan:2011uj}, so the Fourier coefficients in $\vec E(z,s,\rho,0,i)$ are bounded component wise.

On the other hand, the Fourier coefficients of the non-analytic Eisenstein series can be expressed in terms of the Kloosterman-Selberg zeta functions.  Decomposing $z = x+iy$ and using an appropriate unfolding and change of variables, we can see
\begin{align*}
a^{m}_j(y) &= \int_0^h E(z,s,\rho,0,i)_je\left(-m\frac{x}{h}\right)dx\\
&=\int_0^h \sum_{\gamma\in\GG}\left\{\rho(\gamma)^{-1}\right\}_{ji}\frac{y^s}{|cz+d|^{2s}}e\left(-m\frac{x}{h}\right)dx\\
&=\delta_{0,m}\delta_{j,i}y^s +\sum_{\gamma\in\GGstarG} \left\{\rho(\gamma)^{-1}\right\}_{ji}\ \int_{-\infty}^\infty\frac{y^s}{|cz+d|^{2s}}e\left(-m\frac{x}{h}\right)dx\\
&=\delta_{0,m}\delta_{j,i}y^s +\sum_c\frac{1}{|c|^{2s}}\sum_d\left\{\rho(\gamma)^{-1}\right\}_{ji}e\left((m)\frac{d}{ch}\right) \int_{-\infty}^\infty  \frac{y^s}{|x^2+y^2|^{s}}e\left(-m\frac{x}{h}\right)dx\\
&=\delta_{0,m}\delta_{j,i}y^s +y^{1-s}Kl_{0,m}(s,\rho)\int_{-\infty}^\infty \frac{e\left(-myt\right)}{(1+t^2)^s}dt\\
&=Kl_{0,m}(s,\rho)2\pi^s|m|^{s-\frac{1}{2}}\Gamma(s)^{-1}y^{\frac{1}{2}}K_{s-\frac{1}{2}}(2\pi|m|y),
\end{align*}
unless $m=0$, in which case,
\begin{equation*}
a^0_j(y) = \delta_{j,i}y^s +y^{1-s}Kl_{0,0}(s,\rho)\frac{\Gamma\left(\frac{1}{2}\right)}{\Gamma(s)}.
\end{equation*}

Here we have used the modified Bessel function $K_s(z)$.  Moreover, this equation holds for the analytic continuations of $E$ and $Kl$.  There exists an $a>0$, such that for fixed $1-a<\re s$, fixed $z$, and sufficiently large $\im s$ and $z$, the modified Bessel function is bounded below as $\im s\to\infty$ by\cite{Abramowitz:1972vw}.

\begin{equation*}
\left|\frac{\Gamma(s)}{K_{s-\frac{1}{2}}(z)} \right|= O\left(\sqrt{|s|}\right),
\end{equation*}
where the constant is independent of $z$.

This result leads immediately to a bound on the Kloosterman-Selberg zeta function:
\begin{thm*}
Let $\rho$ be a normal multiplier system of even weight with $\alpha_i=0$.  Assume that $m_i\neq 0$,and $1-a<\re s$, then $|Kl_{0,m_j}(s,\rho)_{ji}|=O\left(\sqrt{|s|}\right)$, for sufficiently large $\im s$ where the constant may depend on $\Gamma$, $\rho$, and $m_j$.
\end{thm*}
%\begin{cor*}
%Under the same constraints, $|Kl_{0,m}(s)|_{ij}=O\left(\sqrt{|s|}\right)$, for sufficiently large $\im s$ where the constant may depend on $\Gamma$, $\rho$, $i$, $j$, and $m_j$.
%\end{cor*}
%\begin{proof}
%This follows from the equality $Kl_{\vec m,\vec n}(s) = Kl_{\vec n,\vec m}(s)^{-1}$.
%\end{proof}
%\begin{cor*}
%Let $\rho$ be a normal multiplier system with $n_i\in\alpha_i, m_j\in\alpha_j$.  If $nm=0$ and $n+m\neq 0$, then $\lim Kl_{n_i,m_j}(1)_{<K}=Z_{n_i,m_j}(1)$
%\end{cor*}
%\begin{proof}
%A proof for the scalar case is provided in 
%\end{proof}

\bibliography{VOAsRad}{}
\bibliographystyle{plain}

% Unused citations:~\cite{Bantay:2013hf}

\end{document}